\documentclass{article}
\pdfoutput=1
\usepackage{amssymb,amsmath,amsfonts,amsthm,enumerate,color}
\usepackage[toc,page,title,titletoc,header]{appendix}
\usepackage{natbib}

\usepackage[T1]{fontenc}
\usepackage[utf8]{inputenc}

\usepackage{wrapfig}

\usepackage{amsmath} 
\usepackage{amssymb}  
\usepackage{nccmath}
\usepackage{amsfonts}
\usepackage{mathtools}
\usepackage{bm}
\usepackage{algorithm}
\usepackage{calc}
\usepackage{cite}
\usepackage{color}

\usepackage{siunitx}

\usepackage[shortlabels]{enumitem}
\setlist[enumerate]{nosep}
\setlist[enumerate]{label = (\alph*)}

\usepackage{graphicx}
\usepackage{subfig}
\allowdisplaybreaks

\newtheorem{as}{Assumption}
\newtheorem{lm}{Lemma}
\newtheorem{defin}{Definition}
\newtheorem{pp}{Proposition}
\newtheorem{rem}{Remark}
\newcommand{\tcap}{\text{cap}}
\newcommand{\tilr}{\widetilde{r}}

\newcommand{\sumi}{\sum_{i=1}^n}
\newcommand{\sumj}{\sum\nolimits_{j\neq i}^n}
\newcommand{\sumt}{\sum_{t=0}^{\infty}}
\newcommand{\sumk}{\sum_{k=0}^{\infty}}
\newcommand{\btau}{\overline{\tau}}
\newcommand{\utau}{\underline{\tau}}
\newcommand{\x}{\mathbf{x}}
\newcommand{\y}{\mathbf{y}}
\newcommand{\z}{\mathbf{z}}
\newcommand{\vv}{\mathbf{v}}
\newcommand{\w}{\mathbf{w}}
\newcommand{\ep}{\bm{\epsilon}}
\newcommand{\barx}{\overline{\x}}
\def\T{\mathsf{T}}
\def\X{\mathcal{X}}

\def\F{\widehat{F}}
\def\bone{{\mathbf{1}}}
\def\bzero{{\mathbf{0}}}

\title{Asynchronous Networked Aggregative Games} 

\author{Rongping Zhu
\thanks{
    R. Zhu, J. Zhang and K. You are with the Department of Automation and BNRist, 
    Tsinghua University, Beijing 100084, China. 
    Emails: \texttt{zhurp19@mails.tsinghua.edu.cn, zjq16@mails.tsinghua.edu.cn, youky@tsinghua.edu.cn.}
    }
\and Jiaqi Zhang \footnotemark[1] \and Keyou You \footnotemark[1] \and Tamer~Ba\c{s}ar 
\thanks{
    T. Ba\c{s}ar is with Department of Electrical and Computer Engineering and Coordinated Science Laboratory, University of Illinois at Urbana-Champaign, Urbana, IL 61801 USA.
    Email: \texttt{basar1@illinois.edu.}
    }
}

\begin{document}

\maketitle

\begin{abstract}                          
We propose a fully asynchronous networked aggregative game (Asy-NAG) where each player minimizes a cost function that depends on its local action and the aggregate of all players' actions.
In sharp contrast to the existing NAGs, each player in our Asy-NAG can compute an estimate of the aggregate action at {\em any} wall-clock time by only using (possibly stale) information from nearby players of a directed network.
Such an asynchronous update does not require any coordination among players.
Moreover, we design a novel distributed algorithm with an aggressive mechanism for each player to adaptively adjust the optimization stepsize per update.
Particularly, the slow players in terms of updating their estimates smartly increase their stepsizes to catch up with the fast ones.
Then, we develop an augmented system approach to address the asynchronicity and the information delays between players, and rigorously show the convergence to a Nash equilibrium of the Asy-NAG via a perturbed coordinate algorithm which is also of independent interest.
Finally, we evaluate the performance of the distributed algorithm through numerical simulations.
\end{abstract}

\section{Introduction}

An aggregative game is a Nash game where each player's cost is a function of its action and the aggregate of all players' actions, and has been widely used in flow control \citep{alpcanGametheoreticFrameworkCongestion2002,barreraDynamicIncentivesCongestion2015}, resource allocation \citep{maDecentralizedChargingControl2013, basarPriceBaced2016}, and demand response \citep{liDemandResponseUsing2015,maDecentralizedChargingControl2013}. 
In this work, we consider the \textit{Networked Aggregative Games} (NAGs) over a directed peer-to-peer (P2P) network, where each player individually makes decisions by only using (possibly stale) information from neighboring players of the network, see Fig. \ref{fig:digraph} for an illustrative example.

The NAGs can be categorized as synchronous or asynchronous ones depending on whether players make decisions in a synchronized manner or not. In this paper, we propose an asynchronous NAG (Asy-NAG), which in fact is {\em fully asynchronous}, adopting the asynchronism definition in \citet{PhysRevE.47.2155}.
To elaborate the differences between the synchronous NAG and the Asy-NAG, we use Fig. \ref{fig:digraph}, whose running states for each player are shown in Fig. \ref{fig:time-scale}. A gray bar represents the computation time per update of a player and the red arrows represent information flow directions between players. In a synchronous NAG (Fig. \ref{fig:time-scale-a}), \textit{all} players are essentially synchronized to compute at the same wall-clock time $t(k)$. While in an Asy-NAG (Fig. \ref{fig:time-scale-b}), each player acts \textit{independently} without waiting for other players or delayed messages. Thus, the index $k$ in Fig. \ref{fig:time-scale}(a) increases by one only when all players are synchronized to start their updates, and $k$ in Fig. \ref{fig:time-scale}(b) increases whenever a player starts an update. Importantly, the Asy-NAG does not require any (usually central) coordinator to synchronize the update process and $k$ is not known to players.   

\begin{figure}  \centering
  \includegraphics[width=0.2\linewidth]{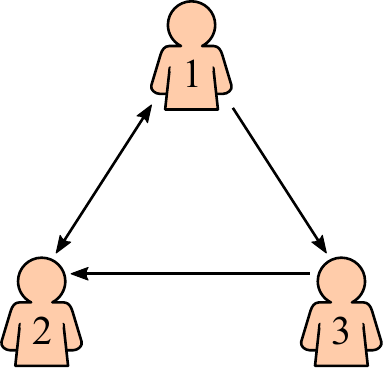}
  \caption{A directed P2P network with three players.}
  \label{fig:digraph}
\end{figure}

\begin{figure}[htp]
  \centering
  \subfloat[A synchronous NAG.]{\label{fig:time-scale-a}\includegraphics[width=0.45\linewidth]{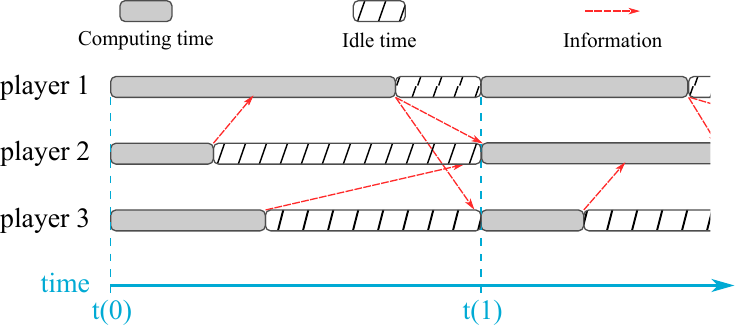}} \hfill
  \subfloat[An Asy-NAG.]{\label{fig:time-scale-b}\includegraphics[width=0.45\linewidth]{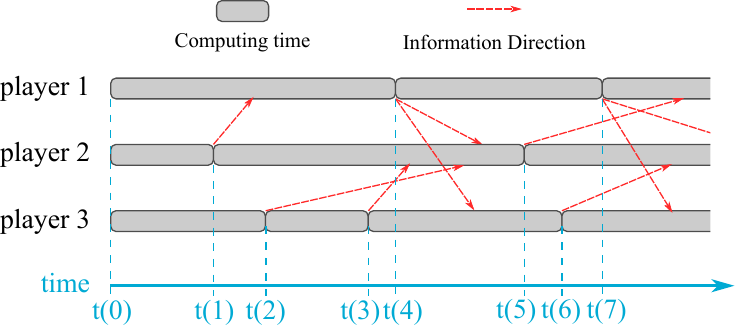}}
  \caption{Comparison of synchronous NAGs and Asy-NAGs.}
  \label{fig:time-scale}
\end{figure}

We believe such Asy-NAGs have not been studied yet, even though they are of more practical importance, in addition to being notoriously difficult to solve.
First, players in a non-cooperative aggregative game can hardly be synchronized to update their actions.
For example, players may pause to operate other tasks, or just remain idle.
In fact, the synchronization of a large-scale NAG can be hard and time-consuming.
Second, due to the asynchronicity and information delays, the updating frequency of each player is unknown and unpredictable. Thus, it is considerably different from the random gossip-based NAG in \citet{koshalDistributedAlgorithmsAggregative2016, salehisadaghianiDistributedNashEquilibrium2016, cenedeseAsynchronousDistributedScalable2020}, requiring the development of a new approach for algorithm design. 

For synchronous NAGs, there are a number of distributed algorithms that converge to a Nash equilibrium (NE) by adopting the consensus-based technique \citep{koshalDistributedAlgorithmsAggregative2016,belgioiosoDistributedGeneralizedNash2019}, ADMM-based methods \citep{salehisadaghianiDistributedNashEquilibrium2019}, the best response dynamics \citep{li1987distributed,lei2020synchronous}, model-free method \citep{frihauf2012nash}, etc.
It is worth stressing that the random gossip-based NAG requires a global coordinator, which is key to their algorithm design and does not satisfy the asynchronism of the Asy-NAGs. Notably, they can be simply modeled as synchronous NAGs over randomly time-varying networks, which is impossible for our Asy-NAG and an augmented network has to be constructed in this work. Thus, the existing distributed algorithms for NAGs cannot be applied to the Asy-NAG.  

In fact, the fully asynchronous setting of this work is not new and has been studied for quite a long history in the parallel and distributed computing problems.
\citet{chazan1969chaotic} first consider it for solving linear equations problems with a parallel system.
\citet{bertsekas1983distributed} and \citet{tsitsiklis1986distributed} proposed asynchronous coordinate algorithms for fixed point problems.
Note that \citet{li1987asymptotic} studied asynchronous non-cooperative games based on the best response dynamics and every player is able to directly access all other players' actions, which is different from the NAGs. Recently, it has been exploited for solving distributed optimization problems \citep{nedic2010convergence, zhangAsynchronousDecentralizedOptimization2019, assran2020asynchronous, sha2020asynchronous}, where players (or computing nodes) cooperate to optimize a common objective function. This is clearly different from NAGs as each player needs to minimize an individual cost function.

Accordingly, we propose a novel distributed algorithm with an aggressive stepsizes scheme to find a NE of the Asy-NAGs.
First, we design an asynchronous perturbed push-sum algorithm to dynamically estimate the aggregate action for each player over the directed network.
When being activated to update, each player uses its locally estimated aggregate action to update its own action.
Second, to address the unpredictable update time instances of players, we propose a novel aggressive scheme for each player to adaptively adjust its optimization stepsize per update.
Specifically, each player maintains a local counter to roughly track the maximum number of updates of its neighboring players, based on which slow players smartly increase stepsizes to catch up with fast ones.
Such an aggressive mechanism is in sharp contrast to the existing works by reducing the stepsize of fast players, and helps to accelerate the convergence to the NE.
A similar idea has been adopted  to address the non-exact convergence issue in distributed optimization problems in \citet{zhangAsySPAExactAsynchronous2019}.

The proposed algorithm is then proved to converge to an NE of the Asy-NAG asymptotically via an augmented network approach.
Particularly, we use the virtual index $k$ to model the delayed information and a sequence of virtual nodes to construct an augmented network.
As a byproduct, we propose a perturbed coordinate pseudo-gradient algorithm for the augmented system, based on which the distributed algorithm for Asy-NAG is proved to converge to the NE.
Finally, we validate the performance of our algorithm with numerical simulations on the renowned Nash-Cournot games over several directed networks. 

The remainder of the paper is organized as follows. In Section \ref{sec:prob}, we introduce the AGs and our Asy-NAG over a P2P network in details. Section \ref{sec:algo} proposes the distributed algorithm for the Asy-NAG. In Section \ref{sec:augmented}, we show how to construct the augmented network and Section \ref{sec:PCPA} provides a novel perturbed coordinate pseudo-gradient algorithm for the AG. In Section \ref{sec:analysis}, we prove the convergence to an NE of the proposed algorithm. Section \ref{sec:numeric} presents the results of numerical experiments. Concluding remarks are drawn in Section \ref{sec:conclusion}. Some technical proofs are included in the appendices.

\textbf{Notations:} Throughout this paper, $\x^\T$ and $X^\T$ denote the transposes of a vector $\x$ and a matrix $X$, respectively. Let $\x^\T \y$ denote the inner product of vectors $\x$ and $\y$, and $\|\x\| = \sqrt{\x^\T \x}$ denote the Euclidean norm of a vector $\x$. We use $[\x_1, \cdots, \x_n]$ to denote the horizontal stack of vectors $\x_1, \ldots, \x_n$, and $\bone_n$ and $\bzero_n$ to denote the $n$-dimensional vectors with all ones and all zeros, respectively. For a matrix $A$, we write $[A]_{ij}$ to denote its $(i,j)$-th element. For a scalar $a$, we let $\lfloor a \rfloor$ denote the largest integer less than $a$. For a set $\mathcal{N}$, let $|\mathcal{N}|$ denote its cardinality, i.e., the number of elements in $\mathcal{N}$. The Minkowski sum of two sets $\mathcal{X}$ and $\mathcal{Y}$ is formed by $\mathcal{X} + \mathcal{Y} := \{\x + \y|\x \in \mathcal{X}, \y\in\mathcal{Y}\}$, and let $a \mathcal{X} \coloneqq \{a \x | \x \in \mathcal{X}\}$. Finally, we use $\Pi_\X$ to denote the Euclidean projection operator over a closed convex set $\X$, i.e., $\Pi_\X(\x) \coloneqq \arg \min_{\z \in \X}\|\x - \z\|$.

\section{Problem Formulation}
\label{sec:prob}

Consider a set of $n$ players indexed by $\mathcal{N} = \{1, \cdots, n\}$. For $i\in\mathcal{N}$, let $f_i(\x_i, \barx)$ denote the cost function of player $i$, where $\x_i \in \X_i \subseteq \mathbb{R}^p$ is the action of player $i$ and $\barx := \frac{1}{n}\sumi \x_i$ is the aggregate of all players' actions.
The objective of player $i$ in an aggregative game (AG) is to solve a local optimization problem
\begin{equation}
\label{eq:prob}
\begin{aligned}
\text{mininize} & \quad \textstyle f_i\left(\x_i, \frac{1}{n} \x_i  + \frac{1}{n} \sumj \x_j\right) \\
\text{subject to} & \quad \x_i \in \X_i,
\end{aligned}
\end{equation}
where $\x_j$, $j \neq i$, $j \in \mathcal{N}$ are fixed,
and to reach an NE among all players, which is defined in precise terms below \citep{bacsar1998dynamic}.
\begin{defin}
  An $n$-tuple of actions
  $$\x^* \coloneqq \left[(\x_1^*)^\T, (\x_2^*)^\T, \cdots, (\x_n^*)^\T\right]^\T$$
  is a Nash equilibrium (NE) if for all $i \in \mathcal{N}$ and $\x_i \in \X_i$,
  \begin{equation*}
    \textstyle f_i\left(\x_i^*, \barx^*\right) \leq f_i\left(\x_i, \frac{1}{n} \x_i + \frac{1}{n} \sumj \x_j^*\right).
  \end{equation*}
\end{defin}
In a centralized AG, a coordinator is needed to gather the actions from all players to compute and broadcast the aggregate action $\overline{\x}$.
On the other hand, in a networked AG (NAG), the players are spatially distributed on a P2P network and each player can only collect information from neighboring players under the topology of the network $\mathcal{G}=(\mathcal{N}, \mathcal{E})$, where $\mathcal{N}$ is the set of players and $\mathcal{E} \subseteq \mathcal{N} \times \mathcal{N}$ is the edge set for information directions, i.e., $(i, j) \in \mathcal{E}$ if and only if $i$ can directly send information to $j$.
We let $\mathcal{N}_{\text{in}}^i = \{j|(j, i) \in \mathcal{E}\} \cup \{i\}$ denote the set of in-neighbors of player $i$, and $\mathcal{N}_{\text{out}}^i = \{j|(i, j) \in \mathcal{E}\}\cup \{i\}$ the set of out-neighbors of $i$.
In this network, player $i$ is able to collect information from players in $\mathcal{N}_{\text{in}}^i$ and broadcast information to players in $\mathcal{N}_{\text{out}}^i$. If $|\mathcal{N}_{\text{in}}^i|\neq |\mathcal{N}_{\text{out}}^i|$ for some $i\in\mathcal{N}$, then $\mathcal{G}$ is unbalanced. 

In a synchronous NAG, all players are synchronized to update their actions per iteration.
In our Asy-NAG, however, players are free to update without waiting for others by only using available information from in-neighbors, which brings significant challenges to the algorithm design and analysis.

\section{The Distributed Algorithm for the Asy-NAG}
\label{sec:algo}

In this section, we first introduce some basic conditions for the existence of an NE of the AG and formulate it as a variational inequality problem. Then, we propose a distributed algorithm for the Asy-NAG where each player asynchronously maintains an estimate of the aggregate action by using information from its in-neighbors, and updates its local action by performing a projected pseudo-gradient step.

\subsection{Variation Inequality Formulation of the AG}
\label{subsec:vi}

The following condition is widely adopted in the context of AGs \citep{belgioiosoDistributedGeneralizedNash2019,liu2020decentralized}.
\begin{as}
  \label{as:f}
  For each $i \in \mathcal{N}$, the action set $\X_i$ is convex and compact.
  Each local objective function $f_i(\x_i, \z)$ is continuously differentiable in $(\x_i, \z)$ over some open set containing $\X_i \times \overline{\X}$ where $\textstyle \overline{\X} = \frac{1}{n} \sumi \mathcal{X}_i$.
  Moreover, $f_i(\x_i, \frac{1}{n} \x_i + \overline{\x}_{-i})$ is convex in $\x_i$ over $\X_i$ for any fixed $\overline{\x}_{-i} = \frac{1}{n}\sum_{j \neq i} \x_j$.
\end{as}
Under Assumption \ref{as:f}, the AG in \eqref{eq:prob} is equivalent to solving a variational inequality problem \citep[Proposition 1.4.2]{facchineiFinitedimensionalVariationalInequalities2003} which is to determine an $\x^* \in \X:=\prod_{i=1}^n \X_i$ such that
\begin{equation}
  \label{eq:optimal}
  (\x - \x^*)^\T\phi(\x^*) \geq 0, \ \forall \x \in \X,
\end{equation}
where $\X$ and $\phi$ are defined as 
\begin{equation*}
  \X \coloneqq \prod_{i=1}^n \X_i, \quad \phi(\x) \coloneqq \begin{bmatrix}
    \nabla_{\x_1} f_1(\x_i, \overline{\x}) \\
    \vdots \\
    \nabla_{\x_n} f_n(\x_n, \overline{\x})
  \end{bmatrix}
\end{equation*}
and $\x \coloneqq [\x_1^\T, \cdots, \x_n^\T]^\T \in \X$.
The mapping $\phi$ is also called a pseudo-gradient mapping. Let
\begin{equation*}
  \label{gradient}
  F_i(\x_i, \z) = \nabla_{\x_i}f_i(\x_i, \z) + \frac{1}{n} \nabla_{\z}f_i(\x_i, \z), \  \forall i \in \mathcal{N}
\end{equation*}
and 
\begin{equation*}
  F(\x, \z) \coloneqq \begin{bmatrix}
    F_1(\x_1, \z) \\
    \vdots \\
    F_n(\x_n, \z)
  \end{bmatrix}.
\end{equation*}
Clearly, $\phi(\x) = F(\x, \overline{\x})$.

\begin{as}[Strict monotonicity]
  \label{as:mono}
  The mapping $\phi(\x)$ is strictly monotone over $\X$ in the sense that
  \begin{equation*}
    (\phi(\x) - \phi(\x'))^\T(\x - \x') > 0, \quad \forall \x, \x' \in \X,\ \x \neq \x'.
  \end{equation*}
\end{as}
\begin{as}[Lipschitz continuity]
  \label{as:FL}
  Each mapping $F_i(\x_i, \z)$ is uniformly Lipschitz continuous over $\X_i \times \overline{\X}$, i.e., there exists some $L>0$ such that for all $\x, \x' \in \X_i$ and $\z, \z' \in \overline{\X}$, 
  $$
  \|F_i(\x_i, \z) - F_i(\x_i', \z')\| \leq L \left\|\begin{bsmallmatrix}
    \x \\ \z
  \end{bsmallmatrix} - \begin{bsmallmatrix}
    \x' \\ \z'
  \end{bsmallmatrix}\right\|.
  $$
\end{as}
\begin{pp}[\citealp{koshalDistributedAlgorithmsAggregative2016}]\label{prop:unique}
Under Assumptions \ref{as:f} and \ref{as:mono}, the AG in \eqref{eq:prob} has a unique Nash equilibrium.
\end{pp}

\subsection{Distributed Algorithm for the Asy-NAG}
\label{subsec:algo}

In this subsection, we propose a push-sum-based distributed algorithm for the Asy-NAG in Algorithm \ref{algo:asypa} where $d_i \coloneqq|\mathcal{N}_{\text{out}}^i| \leq n$ is the out-degree of player $i$ and $\alpha_i$ is the effective stepsize for the local optimization of player $i$. The positive sequence $\{\rho(t)\}$ is used to dynamically adjust the aggressive stepsize $\alpha_i$.

\begin{algorithm}[tbp]
  \caption{The distributed algorithm for the Asy-NAG --- from the view point of player $i$ \label{algo:asypa}}

  \begin{itemize}[leftmargin=*,label={}]
    \item \textbf{Initialization}
      \begin{enumerate}[label=(\roman*)]
      \item \label{initiala}Randomly select an action $\x_i\in \X_i$.
      \item \label{initialb} Let  $\vv_i = \x_i$, $\vv'_i=\vv_i/d_i$ and $y_i=1, y'_i=y_i/d_i$.
      \item \label{initialc} Create local buffers $\mathcal{V}_i$, $\mathcal{Y}_i$ and $\mathcal{L}_i$.
      \item Let $l_i=0$ and broadcast $\vv'_i$, $y'_i$ and $l_i$ to its out-neighbors $\mathcal{N}_{\text{out}}^i$.
    \end{enumerate}
    \item \textbf{Repeat}
    \begin{itemize}[leftmargin=\widthof{aaa},label=$\bullet$]
      \item Keep receiving $\vv'_j$, $y'_j$ and $l_j$ from each in-neighbor player $j\in \mathcal{N}_{\text{in}}^i$, and store them into $\mathcal{V}_i$, $\mathcal{Y}_i$ and $\mathcal{L}_i$, respectively.
      \item If player $i$ is activated for an update, it computes
      \begin{enumerate}
        \item $\w_i \leftarrow\text{sum}(\mathcal{V}_i), y_i\leftarrow\text{sum}(\mathcal{Y}_i)$, $\z_i\leftarrow{\w_i}/{y_i}$, \label{anesa:ps1}
        \item $l_i' \leftarrow \text{max}(\mathcal{L}_i)$, $\alpha_i\leftarrow\sum_{t=l_i}^{l_i'}\rho(t)$, \label{anesa:aggressive}
        \item $\x_i^- \leftarrow \x_i, \x_i \leftarrow \Pi_{\X_i}\left[\x_i-\alpha_i F_i(\x_i, \z_i)\right]$, \label{anesa:update}
        \item $\vv_i \leftarrow \w_i + \x_i - \x_i^-$, \label{anesa:ps2}
        \item $\vv'_i \leftarrow \vv_i/d_i$, $y'_i \leftarrow y_i/d_i$, \label{anesa:ps3}
        \item $l_i\leftarrow l_i'+1$. \label{anesa:l}
      \end{enumerate}
      \item Broadcast $\vv'_i$, $y'_i$ and $l_i$ to every out-neighbors $\mathcal{N}_{\text{out}}^i$, and empty $\mathcal{V}_i, \mathcal{Y}_i, \mathcal{L}_i$.
    \end{itemize}
    \item \textbf{Until} the stopping criterion is satisfied.
  \end{itemize}
\end{algorithm}

The idea of Algorithm \ref{algo:asypa} is very natural. As the synchronous push-sum algorithm \citep{nedicDistributedOptimizationTimeVarying2015} for distributed optimization problems, Line \ref{initialb} in the initialization step is used to solve the unbalancedness issue of the directed network $\mathcal{G}$. However, three local buffers $\mathcal{V}_i,\mathcal{Y}_i,\mathcal{L}_i$ in Line \ref{initialc} are novelly designed to handle the asynchronicity of the Asy-NAG, which is the striking difference from the synchronous NAG. 

In the repeat step, each player keeps receiving (possibly delayed) information from its in-neighbors, and store them into the corresponding buffers.
Due to asynchronous updates of the Asy-NAG, player $i$ may have stored zero, one or multiple receptions from a single neighboring player in each buffer when it is {\em locally} activated to update, which is not the case for the gossip-based NAG \citep{koshalDistributedAlgorithmsAggregative2016, salehisadaghianiDistributedNashEquilibrium2016, cenedeseAsynchronousDistributedScalable2020}.
Instead of only using the latest reception from buffers, all data in buffers will be used to compute a new update, see Lines \ref{anesa:ps1}-\ref{anesa:l}, where $\text{sum}(\cdot)$ and $\text{max}(\cdot)$ take, respectively, summation and maximization over all elements in the input set. Then, it broadcasts the updated vectors to its out-neighbors and empties the buffers. This implies that the number of receptions in each buffer is usually limited. Interestingly, buffers are essentially not needed in practice since both summation and maximization in Lines \ref{anesa:ps1}-\ref{anesa:aggressive} can be done recursively. For example, we can simply keep ${l}_i'$ and only update its value to $l_j$ if a new $l_j>{l}_i'$ has been received from an in-neighbor.  

Our key idea for the Asy-NAG is essentially captured in Lines \ref{anesa:ps1}-\ref{anesa:aggressive}. Informally speaking, ${l}_i$ is designed to roughly track the number of updates in each player and ${l}_i'$ returns the maximum number of updates among its in-neighbors. If player $i$ finds itself update too slow, e.g., the gap between ${l}_i$ and ${l}_i'$ is large, it increases its stepsize $\alpha_i$ to compensate for slow updates. Such an aggressive scheme is in sharp contrast with the existing technique to reduce the stepsize of the fast player, hoping to achieve a faster convergence rate, and was initially proposed in our previous work \citep{zhangAsySPAExactAsynchronous2019} for solving the distributed optimization problem. The stopping criterion can be chosen such that the update of the local action and aggregate estimate remain small for a number of consecutive iterations.

The rest of this paper is devoted to proving the convergence of Algorithm \ref{algo:asypa} to an NE of the Asy-NAG. We first develop an augmented network approach to address the asynchronicity and information delays in	 Algorithm \ref{algo:asypa}, under which we obtain a synchronous coordinate-wise update over the virtual network. To prove its convergence to an NE, we then propose a coordinate pseudo-gradient algorithm, which generalizes Algorithm \ref{algo:asypa} and is of independent interest. Indeed, the fully asynchronous setting has been adopted for distributed optimization \citep{zhangAsySPAExactAsynchronous2019}  and the block coordinate optimization \citep{hannah2018a2bcd}, which however cannot be used for the AG as each player needs to minimize an individual objective function.


\section{An Augmented Network Approach for the Asy-NAG}
\label{sec:augmented}

In this section,  we construct an augmented network where a sequence of virtual nodes is introduced for each player to model the delayed information, based on which Algorithm \ref{algo:asypa} appears to be synchronous over the virtual network. Note that the proofs in \citet{koshalDistributedAlgorithmsAggregative2016, salehisadaghianiDistributedNashEquilibrium2016, cenedeseAsynchronousDistributedScalable2020} cannot be directly applied here. 

\begin{as}
  \label{as:graph}
  \hfill
  \begin{enumerate}
    \item \textbf{(Strong connectivity)} $\mathcal{G}$ is strongly connected, i.e., there is a directed path between any two players.
    \item \textbf{(Bounded information delays)} For any edge $(i, j) \in \mathcal{E}$, the time-varying information delay from player $i$ to player $j$ is uniformly bounded by a positive constant $\tau > 0$.
    \item \textbf{(Bounded activation intervals)} Let $t_i$ and $t_i^+$ be two consecutive activation times of player $i$; then, there exist $\underline{\tau} > 0$ and $\overline{\tau} < \infty$ such that $\underline{\tau} \leq |t_i - t_i^+| \leq \overline{\tau}$.
  \end{enumerate}
\end{as}

The strong connectivity ensures each player to participate in the decision making process. The bounded communication delays assumption is common in the literature, see e.g., \citet{yi2019asynchronous, lei2020synchronous}.
The bounded activation interval is practically satisfied and is almost necessary, e.g., if $\overline{\tau}=\infty$, it suggests that some player(s) does(do) not participate in the decision-making process. The above constants are introduced for theoretical analysis and are not known by any player.

Under Assumption \ref{as:graph}(c), the activation time instances for updating must be discrete. Thus, we 
let $\mathcal{T} = \{t(k)\}_{k\geq 0}$ be an increasing sequence of the activation time instances of all players, i.e., $t \in \mathcal{T}$ if and only if there is at least one player being activated at time $t$.
Let $\mathcal{T}_i \subseteq \mathcal{T}$ be the set of activation times of player $i$.
Lemma \ref{lm:delays} is key to the construction of the augmented network.

\begin{lm}[\citealp{zhangAsySPAExactAsynchronous2019}]
  \label{lm:delays}
  The following statements hold.
  \begin{enumerate}
    \item Under Assumption \ref{as:graph}(b), letting $b_1 = (n-1)\lfloor\btau/\utau\rfloor+1$, each player is activated at least once within the time interval $(t(k), t(k+b_1)]$. Moreover, let $\mathcal{G}(k) = (\mathcal{N}, \mathcal{E}(k))$ be the activated network at $t(k)$, where $(i, j) \in \mathcal{E}(k)$ if $t(k) \in \mathcal{T}_i$ and $(i, j) \in \mathcal{E}$; then the union of networks $\bigcup_{t=k}^{k+b_1}\mathcal{G}(t)$ is strongly connected for any $k$.
    \item Under Assumptions \ref{as:graph}(b) and (c), letting $b_2 = n\lfloor\tau/\utau\rfloor+1$ and $b = b_1+b_2$, the message sent from player $i$ at $t(k)$ can be received by player $j$ before $t(k+b_2)$ and used for computing an update before $t(k+b)$ for any $k$ and $(i, j) \in \mathcal{E}$.
    \item Under Assumption \ref{as:graph}, and with $b$ defined as in (b) above, $|l_i(k) - l_j(k)| \leq nb$ and $0 \leq l_i(k+1) - l_i(k) \leq nb + 1$ for any $i, j \in \mathcal{N}$ and $k$, where $l_i(k)$ is the value of $l_i$ at $t(k)$.
  \end{enumerate}
\end{lm}

To construct an augmented network, we introduce $b$ virtual nodes  for each player. For player $i$, let $\{i^{(1)}, \cdots i^{(b)}\}$ be the corresponding sequence of virtual nodes. Denote by $\widetilde{\mathcal{N}}$ the set containing all players and virtual nodes. As for the edge set $\widetilde{\mathcal{E}}(k)$, edges $(i^{(1)}, i)$, $(i^{(2)}, i^{(1)})$, $\cdots$, $(i^{(b)}, i^{(b-1)})$ are always included in $\widetilde{\mathcal{E}}(k)$ for all $i$. If player $i$ is activated to update at $t(k)$, then one of the edges $(i, j)$, $(i, j^{(1)})$, $(i, j^{(2)})$, $\cdots$, $(i, j^{(b)})$ is included in $\widetilde{\mathcal{E}}$, depending on the information delay between players $i$ and $j$ at $t(k)$. Specifically, if the information from player $i$ is received by player $j$ during $(t(k + s), t(k + s + 1)]$ for some $s \in \{0, \cdots, b\}$, then $(i, j^{(s)}) \in \widetilde{\mathcal{E}}(k)$. Note that we have adopted the convention $j^{(0)} = j$. 
Then, $\widetilde{\mathcal{G}}(k) = (\widetilde{\mathcal{N}}, \widetilde{\mathcal{E}}(k))$ is a time-varying augmented network of $\mathcal{G}$ at $t(k)$. A scenario, related to the network of Fig. \ref{fig:digraph}, is depicted in Fig. \ref{fig:augmented}.

\begin{figure}[t!]
  \centering
  \includegraphics[width=0.4\linewidth]{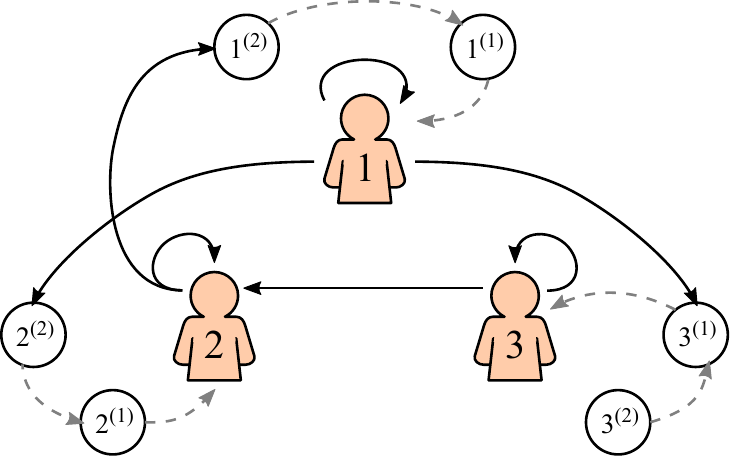}
  \caption{An augmented network of Fig. \ref{fig:digraph} with $b=2$ at $t(k)$. Player 3's message will be received by player 2 during the time interval $(t(k), t(k+1)]$, and $(3, 2) \in \mathcal{E}(k)$. Player 3 will use information from $3^{(1)}$ and itself to compute update at $t(k+1)$. Suppose that player 2 sends message to player 1 without delay but player 1 will not be activated until $t(k+3)$; then $(2, 1^{(2)}) \in \mathcal{E}(k)$. The same case applies to player 1.}
  \label{fig:augmented}
\end{figure}

Now, we reformulate Algorithm \ref{algo:asypa} as a synchronous algorithm over the time-varying virtual networks $\{\widetilde{\mathcal{G}}(k)\}$.
We denote the latest value of $\x_i$ before time $t(k)$ by $\x_i(k)$.
We adopt similar notations for $\vv_i$, $\w_i$, $\z_i$, $y_i$ and $l_i$.
Letting $\mathcal{N}^{(u)} = \{i^{(u)} | i\in\mathcal{N}\}$ for $u = 1, 2, \cdots, b$, we enumerate the players before virtual nodes, e.g., $\widetilde{\mathcal{N}} = \{\mathcal{N}, \mathcal{N}^{(1)}, \cdots, \mathcal{N}^{(b)}\}$. Let $\tilde{n} = (b+1)n$.
Every virtual node $i$ maintains $\vv_i(k)$, $\w_i(k)$ and $y_i(k)$.
Let the augmented state vector be
\begin{equation*}
  \begin{aligned}
    \widetilde{V}(k) &= [\vv_1(k), \cdots, \vv_n(k), \vv_{n+1}(k), \cdots, \vv_{\tilde{n}}(k)]^\T, \\
    \widetilde{W}(k) &= [\w_1(k), \cdots, \w_n(k), \w_{n+1}(k), \cdots, \w_{\tilde{n}}(k)]^\T, \\
    \widetilde{\y}(k) &= [y_1(k), \cdots, y_n(k), y_{n+1}(k), \cdots, y_{\tilde{n}}(k)]^\T,
  \end{aligned}
\end{equation*}
with $\vv_i(0) = \bzero_p$ and $y_i(0) = 0$ for $i = n + 1, \cdots, \tilde{n}$.

Then, Algorithm \ref{algo:asypa} is compactly written as
\begin{equation}
  \label{eq:pushsum}
  \begin{aligned}
    \widetilde{W}(k+1) &= \widetilde{A}(k)\widetilde{V}(k), \\
    \widetilde{\y}(k+1) &= \widetilde{A}(k)\widetilde{\y}(k), \\
    \z_i(k+1) &= \frac{\w_i(k + 1)}{y_i(k+1)}, \quad \forall i \in \mathcal{N} \\
    \x_i(k+1) &= \begin{cases}
      \Pi_{\X_i}\left[\x_i(k) - \alpha_i(k)F_i(\x_i(k), \z_i(k+1))\right], & \text{if } i \in \mathcal{N} \text{ and } t(k) \in \mathcal{T}_i, \\
      \x_i(k), & \text{if } i \in \mathcal{N} \text{ and } t(k) \notin \mathcal{T}_i,
      \end{cases} \\
    \widetilde{V}(k+1) &= \widetilde{W}(k+1) + \Delta X(k),
  \end{aligned}
\end{equation}
where
\begin{equation*}
  \label{eq:deltaX}
  \Delta X(k) = [\x_1(k+1) - \x_1(k), \cdots, \x_n(k+1) - \x_n(k), \bzero_p, \cdots, \bzero_p]^\T.
\end{equation*}
$\widetilde{A}(k)$ is given as
$$
  [\widetilde{A}(k)]_{ij}
  = \begin{cases}
    1/d_j, & \text{if } j\in\mathcal{N}, t(k) \in \mathcal{T}_j, i=un+i', \tau_{ji'}(k) = u\\
    1, &\text{if } j\in\mathcal{N}, t(k) \notin\mathcal{T}_j, \text{ and } i = j \\
    1, &\text{if } j \notin\mathcal{N}, \text{ and } i = j - n \\
    0, &\text{otherwise}
  \end{cases}
$$
where $d_j = |\mathcal{N}_{\text{out}}^j|$ is the out-degree of player $j$ and $\tau_{ji'}(k) \geq 0$ is the information delay from player $j$ to node $i'$ at $t(k)$. If each player can have immediate access to its own action, then $\tau_{jj}(k) = 0$ for all $j\in\mathcal{N}$ and $k \geq 0$. If $(j, i) \notin \mathcal{E}$, we have $\tau_{ji}(k) = \infty$ for all $i\in\mathcal{N}$.
Note that $\widetilde{A}(k)$ is column-stochastic and $[\widetilde{A}(k)]_{ii} > 0$ for all $i \in \mathcal{N}$ and $k \geq 0$. Thus, $y_i(k) > 0$ always holds and the division in \eqref{eq:pushsum} is well-defined.

\section{The Perturbed Coordinate Pseudo-gradient Algorithm}
\label{sec:PCPA}

Inspired by the block-coordinate method for a single objective function \citep{nesterov2012efficiency}, we first propose a novel perturbed coordinate pseudo-gradient algorithm (PCPA) which allows partial updates of players via perturbed pseudo-gradients.
Subsequently, we prove its convergence, the results being not only of independent interest, but also critical to the proof of the convergence of Algorithm \ref{algo:asypa} in the next section.

Consider the AG in \eqref{eq:prob} with a single player to update with a perturbed pseudo-gradient per iteration.
We denote the sequence of updating players by $\{s(0), s(1), \cdots\}$, where $s(k) \in \mathcal{N}$, i.e., player $s(k)$ is selected to update at the $k$th iteration.
An integer-valued sequence $\{r_i(k)\}$ is defined recursively, i.e.,
$$
r_i(k+1) = \begin{cases}
  r_i(k) + \Delta r_i(k), & \text{if } s(k) = i, \\
  r_i(k), & \text{otherwise,}
\end{cases}
$$
and $r_i(0) = 0$ for all $i \in \mathcal{N}$.  In light of the above sequence, we propose the following PCPA
\begin{equation}
  \label{eq:pcpa}
  \x_i(k+1) = \begin{cases}
    \Pi_{\X_i}\left[\x_i(k) - \alpha_i(k) \widehat{F}_i(k)\right], & \text{if } s(k) = i, \\
    \x_i(k), & \text{otherwise},
  \end{cases}
\end{equation}
where
$$
  \begin{aligned}
    \alpha_i(k) &= \textstyle \sum_{t=r_i(k)}^{r_i(k+1)-1}\rho(t), \\
    \widehat{F}_i(k) &= F_i(\x_i(k), \overline{\x}(k)) + \ep_i(k),
  \end{aligned}
$$
and $\ep_i(k)$ is the perturbation of the gradient in \eqref{gradient}. We make the following assumption.
\begin{as}
\label{as:pcpa}
\hfill
\begin{enumerate}
  \item There exists an integer $\sigma_1 \geq 0$ such that $\mathcal{N} \subseteq \{s(k+1), \cdots, s(k+\sigma_1)\}$.
  \item There exists an integer $\sigma_2 \geq 0$ such that $|r_i(k) - r_j(k)| \leq \sigma_2$ and $\Delta r_i(k) \leq \sigma_2$ for all $i, j \in \mathcal{N}$ and $k \geq 0$.
  \item The sequence $\{\rho(t)\}$ is chosen to satisfy that $0 < \rho(t + 1) \leq \rho(t)$, $\sumt \rho(t) = \infty, \sumt \rho^2(t) < \infty$.
  \item $\sum_{k=0}^{\infty}\rho(k)\|\ep_i(k)\| < \infty$ for all $i \in \mathcal{N}$.
\end{enumerate}
\end{as}

Assumption \ref{as:pcpa}(a) guarantees that each player updates at least once in a bounded time interval, which is different from the randomized block-coordinated method \citep{richtarik2014iteration}.
Assumption \ref{as:pcpa}(b) ensures the boundedness of the updating time difference  between two players. It is obvious that Assumptions \ref{as:pcpa}(a)-(b) hold for the Asy-NAG with $\sigma_1 = b_1$ and $\sigma_2 = nb + 1$, where $b_1$ and $b$ are defined in Lemma \ref{lm:delays}.
Unlike \citet{yousefian2013distributed} and \citet{bottou2010large}, the perturbation $\ep_i(k)$ is not stochastic.
Instead, it should be controlled by  non-increasing sequences in Assumptions \ref{as:pcpa}(c)-(d), which is essential to the convergence of the PCPA.

\begin{pp}\label{prop:pcpa}
  Suppose that Assumptions \ref{as:f}-\ref{as:FL} and \ref{as:pcpa} hold.
  Then, the sequence $\{\x_i(k)\}$, $i = 1, \cdots, n$, generated by the PCPA in \eqref{eq:pcpa} converges to the unique NE of the AG in \eqref{eq:prob}.
\end{pp}
\begin{proof}
  See Appendix \ref{ap:pcpa}.
\end{proof}

\section{Convergence of Algorithm \ref{algo:asypa}}
\label{sec:analysis}
Under a non-increasing sequence $\{\rho(t)\}$ in Assumption \ref{as:pcpa}(c), we show that each player is able to asymptotically track the aggregate action $\overline{\x}(k)$.

\begin{lm}\label{lm:consensus}
	Under Assumptions \ref{as:f}, \ref{as:FL}, \ref{as:graph} and \ref{as:pcpa}(c), let $\{\z_i(k)\}$, $i = 1, \cdots, n$, be the sequence generated by the augmented network formulation \eqref{eq:pushsum} of Algorithm \ref{algo:asypa}. Then, the following statements hold.
	\begin{enumerate}
		\item For all $k\geq 0$,
    \begin{equation*}
      \textstyle\left\|\z_i(k+1)-\overline{\x}(k)\right\| \leq B \left(\lambda^k nM +\sum_{t=0}^{k}\lambda^{k-t}\Delta(t)\right),
    \end{equation*}
    where $\Delta(t) = \sumi \|\x_i(t + 1) - \x_i(t)\|$ and the constants are defined as
    $B = 8n^{nb}(1+n^{nb})$,
    $\lambda = (1-\frac{1}{n^{nb}})^{\frac{1}{nb}}$,
    and $M = \max_{i\in \mathcal{N}, \x_i \in \X_i}\|\x_i\|$.
    \item $  \lim_{k\to\infty}\|\z_i(k+1)-\overline{\x}(k)\|=0,\ \forall i\in\mathcal{N}.$
		\item \label{lm:consensusD} 
    $
      \sum_{k=0}^{\infty}\rho(k)\|\z_i(k+1)-\overline{\x}(k)\|<\infty,\ \forall i\in\mathcal{N}.
    $
	\end{enumerate}
\end{lm}
\begin{proof}
  See Appendix \ref{ap:consensus}.
\end{proof}

The convergence proof of Algorithm \ref{algo:asypa} then directly follows from Proposition \ref{prop:pcpa}.

\begin{pp}\label{prop:asypa}
  Suppose that Assumptions \ref{as:f}-\ref{as:graph} hold and the sequence $\{\rho(t)\}$ is chosen to satisfy Assumption \ref{as:pcpa}(c).
  Then, the sequence $\{\x_i(k)\}$, $i = 1, \cdots, n$, generated by the augmented network formulation \eqref{eq:pushsum} of Algorithm \ref{algo:asypa} converges to the unique NE of the AG in \eqref{eq:prob}.
\end{pp}

\begin{proof}
  Let
  $$ \ep_i(k) = F_i(\x_i(k), \z_i(k + 1)) - F_i(\x_i(k), \overline{\x}(k)).$$
  Under the Lipschitz continuity of $F_i$, we have
  $$ \|\ep_i(k)\| \leq L \|\z_i(k + 1) - \overline{\x}(k)\|.$$
  Thus, it follows from Lemma \ref{lm:consensus} that $\lim_{k\to\infty}\|\ep_i(k)\| = 0$ and $\sumk \rho(k)\|\ep_i(k)\|<\infty$. Let $\mathcal{A}(k)$ denote the set of updating players at time $t(k)$, i.e., $\mathcal{A}(k) = \{i|t(k)\in\mathcal{T}_i, i\in\mathcal{N}\}$. It follows from Algorithm \ref{algo:asypa} that
  \begin{equation}
    \label{eq:asypaAsPcpa}
    \x_i(k+1) = \begin{cases}
      \Pi_{\X_i}\left[\x_i(k) - \alpha_i(k)\widehat{F}_i(k)\right], & \text{if } i\in\mathcal{A}(k), \\
      \x_i(k), & \text{otherwise},
    \end{cases}
  \end{equation}
  where $\alpha_i(k) = \sum_{t=l_i(k)}^{l_i(k+1)-1}\rho(t)$, and the perturbed gradient $\widehat{F}_i(k) = F_i(\x_i(k), \overline{\x}(k)) + \ep_i(k)$. Then, two cases are considered.

  \textbf{Case 1:} If $\mathcal{A}(k)$ is a singleton for all $k \geq 0$, it follows from Proposition \ref{prop:pcpa} that $\x_i(k)$ converges to the NE.

  \textbf{Case 2:} If for some $k$, $\mathcal{A}(k)$ includes multiple elements, say $i_1, \cdots, i_m$, where $m = |\mathcal{A}(k)|\leq n$, we can incrementally select the players via $m$ iterations. 
  
  Specifically, letting $\x^{(0)} = \x(k)$ and $\x^{(m)} = \x(k+1)$, we can rewrite \eqref{eq:asypaAsPcpa} as
  \begin{equation*}
    \x_i^{(u)} = \begin{cases}
      \Pi_{\X_i}\left[\x_i^{(u-1)} - \alpha_i(k)\widehat{F}_i^{(u-1)}\right], & \text{if } i = i_u \\
      \x_i^{(u-1)}, & \text{otherwise}
    \end{cases}
  \end{equation*}
  for $u = 1, \cdots, m$, where
  $$
    \begin{aligned}
      \widehat{F}_i^{(u-1)} &= F_i(\x_i^{(u-1)}, \overline{\x}^{(u-1)}) + \ep_i^{(u-1)}, \\
      \ep_i^{(u-1)} &= \ep_i(k) + F_i(\x_i(k), \overline{\x}(k)) - F_i(\x_i^{(u-1)}, \overline{\x}^{(u-1)}).
    \end{aligned}
   $$
This implies that
  $$
    \textstyle
    \begin{aligned}
      & \|F_i(\x_i(k), \overline{\x}(k)) - F_i(\x_i^{(u-1)}, \overline{\x}^{(u-1)})\| \\
      & \leq \|F_i(\x_i(k), \overline{\x}(k)) - F_i(\x_i(k), \overline{\x}^{(u-1)})\| + \|F_i(\x_i(k), \overline{\x}^{(u-1)}) - F_i(\x_i^{(u-1)}, \overline{\x}^{(u-1)})\| \\
      & \leq \frac{L}{n} \sum_{s=1}^{u-1}\|\x_{i_s}(k+1) - x_{i_s}(k)\| + L \|\x_{i_{u-1}}(k+1) - x_{i_{u-1}}(k)\|
    \end{aligned}
  $$
  for $u = 2, \cdots, m$. Moreover, it holds that
  $$
      \|\x_{i_s}(k+1) - \x_{i_s}(k)\| \leq \alpha_{i_s}(k)\|\widehat{F}_i^{(s-1)}\| \leq \alpha_{i_s}(k)(C + \ep_i(k)),
  $$
  where
  $\textstyle C = \max_{\x\in \mathcal{X}} \max_{i \in \mathcal{N}}F_i(\x_i, \overline{\x})<\infty$
  due to the compactness of $\mathcal{X}$.
  Thus, it is easy to verify that $\sumk\rho(k)\|\sum_u\ep_i^{(u)}(k)\| < \infty$. 
  Finally, the convergence of Algorithm \ref{algo:asypa} follows from Proposition \ref{prop:pcpa} again.
\end{proof}

For the Asy-NAGs with strictly monotone pseudo-gradients, it does not seem to be possible to explicitly evaluate the convergence rate of Algorithm \ref{algo:asypa}, and thus we resort to numerical experiments.
In the next section, numerical results illustrate that such a rate is related to network connectivity and the level of asynchronicity.
\begin{rem} It is interesting to note that the convergence in Proposition 3 is in the deterministic sense, rather than in the stochastic sense as in \citet{koshalDistributedAlgorithmsAggregative2016, salehisadaghianiDistributedNashEquilibrium2016, cenedeseAsynchronousDistributedScalable2020}. 
\end{rem}

\section{Numerical Experiments}
\label{sec:numeric}
In this section, we validate the performance of Algorithm \ref{algo:asypa} on networked Nash-Cournot games.
The convergence to the NE is confirmed under different network topologies with different numbers of players.
We also compare Algorithm \ref{algo:asypa} with its synchronous counterpart. Moreover, we validate the advantages of using the aggressive stepsize in Line (b) of Algorithm \ref{algo:asypa}. 

The Nash-Cournot game model is adopted from \citet{koshalDistributedAlgorithmsAggregative2016}. Specifically, consider $n$ firms competing over $L$ markets. Let firm $i$'s production and sales at market $l$ be denoted by $g_{il}$ and $s_{il}$, respectively, while its cost of production at market $l$ is $c_{il}(g_{il})$ and defined as
$$
c_{il}(g_{il}) = a_{il}g_{il} + b_{il}g_{il}^2,
$$
where $a_{il}$ and $b_{il}$ are positive parameters for firm $i$.
The revenue of firm $i$ at market $l$ is $p_{l}(s_l)s_{il}$, where $p_{l}(s_l)$ denotes the sale price at market $l$ and $s_{l} = \sumi s_{il}$ is the total sales at market $l$.
The sale price function captures the inverse demand function and is defined as
$$
  p_{l}(s_l) = d_l - s_l,
$$
where $d_l$ is the overall demand at market $l$. Firm $i$'s production at market $l$ is capacitated by $\tcap_{il}$. The transportation costs between any two markets are set to zero. Let $\x_{il} = (g_{il}, s_{il})$ for all $l = 1, \cdots, L$,  $\x_i = (\x_{i1}, \cdots, \x_{iL})^\T$, and $\x = (\x_1^\T, \cdots, \x_n^\T)^\T$, firm $i$'s optimization problem is then given by the following:
$$
  \begin{aligned}
    \text{minimize} &\quad f_i(\x_i, \overline{\x}) = \sum_{l=1}^L \left(c_{il}(g_{il}) - s_{il}\cdot p_l(s_l)\right) \\
    \text{subject to} &\quad g_{il},s_{il}\geq 0,\ g_{il}\leq \text{cap}_{il}, \\
    &\quad \sum_{l=1}^L g_{il} = \sum_{l=1}^L s_{il}.
  \end{aligned}
$$
Let $\X_i$ denote the constraint set for firm $i$.
Obviously, for each $i$, both $\X_i$ and $f_i$ satisfy Assumption \ref{as:f}, and  the pseudo-gradient mapping $\phi$ satisfies Assumptions \ref{as:mono} and \ref{as:FL}.
Thus, it follows from Proposition \ref{prop:unique} that the Nash-Cournot game admits a unique NE.
Besides, $\phi$ is also \textit{strongly monotone}\footnote{In the experiments, let $\rho(t)=\rho$ to achieve a linear convergence rate, which is to be studied in the future work.}, i.e., there exists a constant $\mu > 0$ such that for all $\x, \x' \in \X$,
\begin{equation*}
  (\x - \x')^\T(\phi(\x) - \phi(\x')) \geq \mu \|\x - \x'\|^2.
\end{equation*}
The Nash-Cournot game consists of $n \in \{5, 10, 20, 30\}$ firms over $L = 10$ markets. The parameters $a_{il}$, $b_{il}$, and $d_l$ are drawn from a uniform distribution. Specifically, for all $i\in \mathcal{N}$ and $l\in\{1, \cdots, L\}$, we set $a_{il} \sim U(2, 12)$, $b_{il} \sim U(2, 3)$, where $U(u_1, u_2)$ denotes the uniform distribution over an interval $[u_1, u_2]$. The demand $d_l$ is drawn from $U(90, 100)$. Moreover, we set the production capacities as $\tcap_{il} = 500$.

We adopt the Message Passing Interface (MPI) on a multi-core server to simulate the P2P network.
Particularly, the MPI uses $n$ cores to denote the players, and the communication is performed between neighboring cores in the predefined network.
To simulate a heterogeneous environment, the computation time of player $i$ is sampled from an exponential distribution $\text{exp}(1 / \mu_i) \si{\ms}$.
For player $i$, $\mu_i$ is set as $1 + |\overline{\mu}_i|$ where $\overline{\mu}_i$ follows the standard normal distribution $N(0, 5^2)$.
The information delays are sampled from $\text{exp}(1/5) \si{\ms}$.

\begin{figure}[tbp]
  \centering
  \subfloat[Cycle]{\includegraphics[width=.15\linewidth]{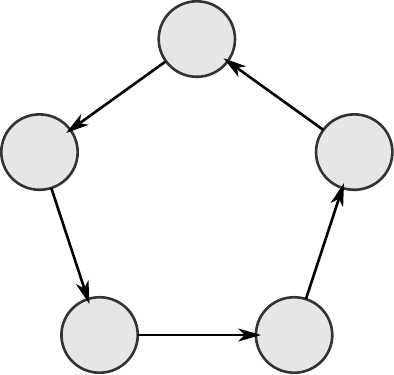}}
  \hfill
  \subfloat[Star]{\includegraphics[width=.15\linewidth]{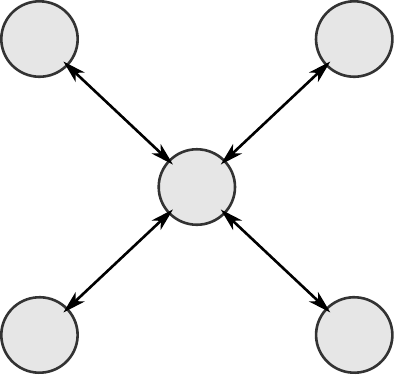}}
  \hfill
  \subfloat[Log]{\includegraphics[width=.15\linewidth]{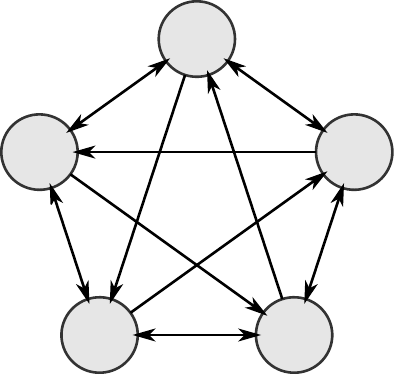}}
  \hfill
  \subfloat[Complete]{\includegraphics[width=.15\linewidth]{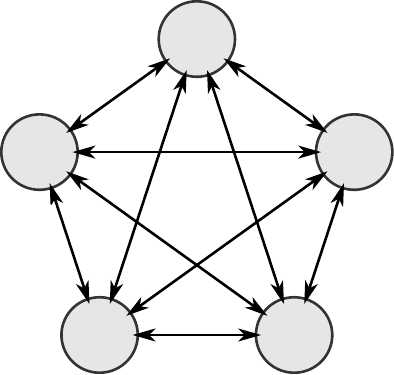}}
  \caption{An illustration of network structures used in experiments.}
  \label{fig:topos}
\end{figure}

In the figures to follow, we plot the trajectories of the normalized suboptimality gap $$\frac{\|\x(t) - \x^*\|_{\infty}}{\|\x^*\|_{\infty}}.$$
The curves are averaged over 50 Monte-Carlo simulations with randomly initiated points and sampled delays.

Firstly, we consider an Asy-NAG with $n = 20$ players over four network structures which are described below and illustrated in Figure \ref{fig:topos}.
\begin{itemize}
  \item \textit{Cycle}: every player has only one in-neighbor and one out-neighbor.
  \item \textit{Log}: player $i$ sends information to players $\text{mod}(2^j + i, n) + 1$, $0 \leq j < \log_2n$.
  \item \textit{Star}: a player is able to send and receive information from any other player.
  \item \textit{Complete}: every player sends information to every other player.
\end{itemize}
Figure \ref{fig:numerics}(a) confirms the convergence of Algorithm \ref{algo:asypa} to the NE over different network structures.
Specifically, the convergence rate over the \textit{complete} network is the fastest while that over the \textit{cycle} is the slowest, which coincides with the fact that the \textit{complete} network has the best connectivity.
Note that although the \textit{star} network seems denser than the \textit{log} one when $n = 20$, the algorithm is bottlenecked by the central player, thus the convergence rate is slower.

Secondly, we consider Asy-NAGs with $n = 5, 10, 20, 30$ players over the \textit{log} network.
Figure \ref{fig:numerics} demonstrates the convergence of Algorithm \ref{algo:asypa}.
The asynchronicity of the Asy-NAG increases with the number of players, which decelerates the convergence rate of the estimated aggregate action.
Thus, Algorithm \ref{algo:asypa} requires more time to converge when the number of players increases.

\begin{figure}[tbp]
  \centering
  \subfloat[]{\includegraphics[width=.4\linewidth]{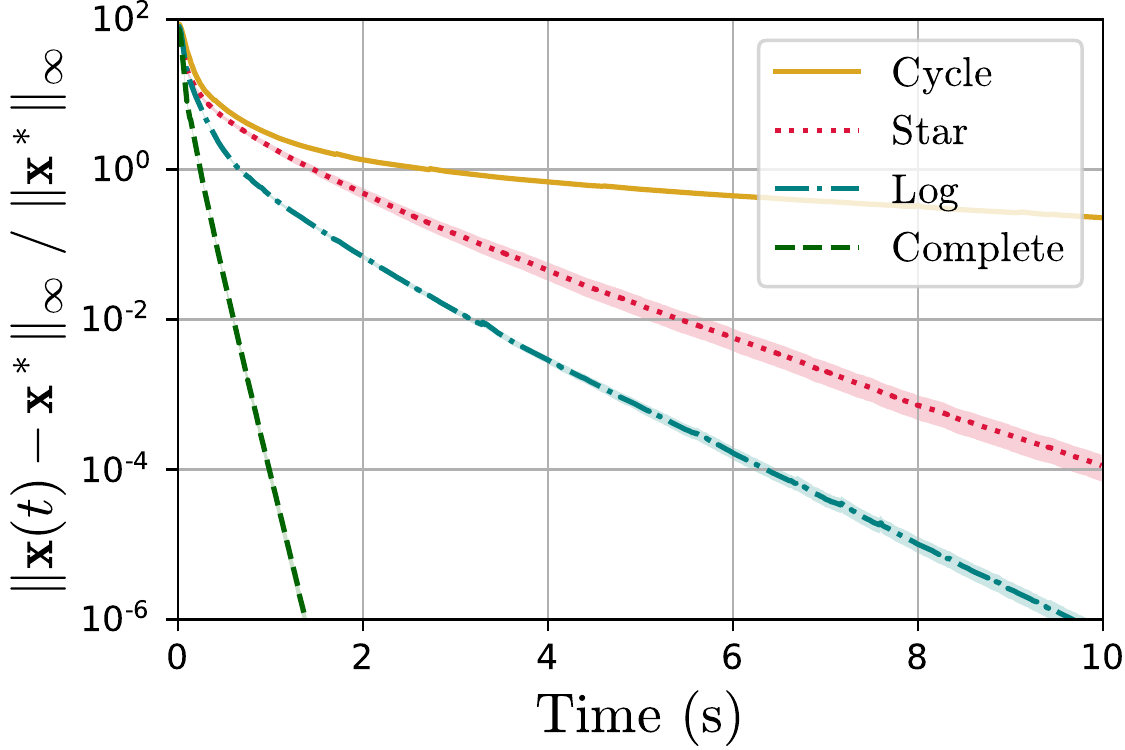}} \hfill
  \subfloat[]{\includegraphics[width=.4\linewidth]{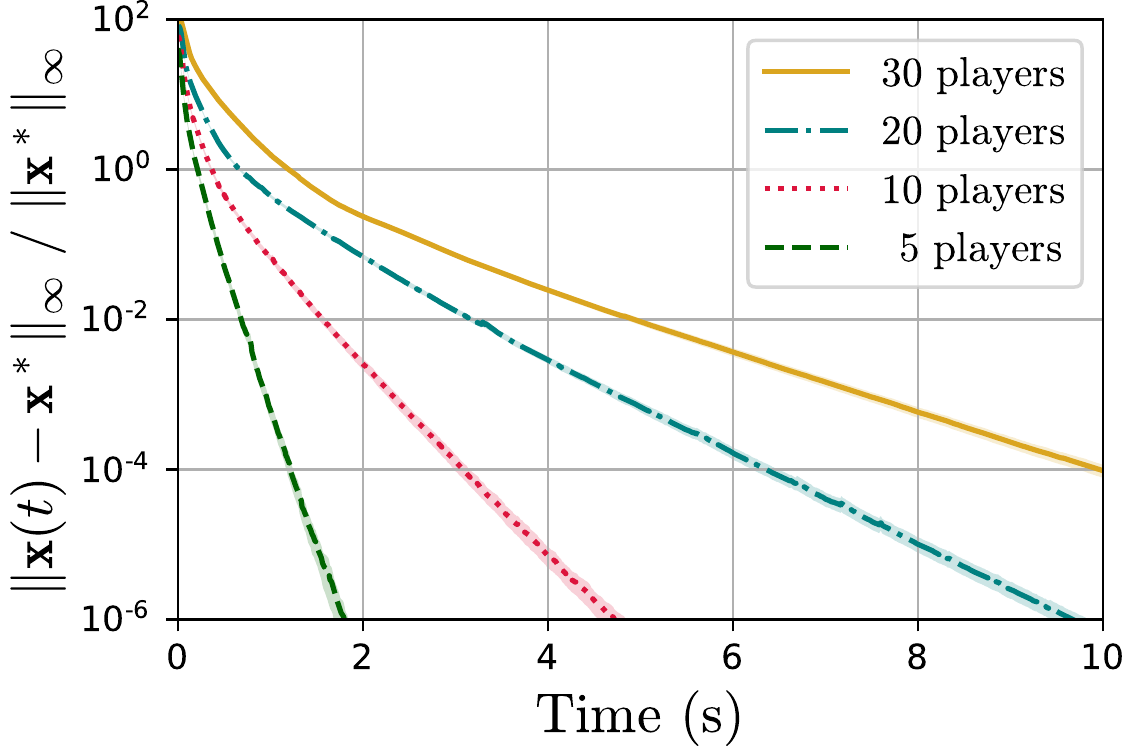}}
  \caption{The suboptimality gap versus the running time. (a) The trajectories of the suboptimality gap for Algorithm \ref{algo:asypa} with 20 players over different network structures. (b) The trajectories of the suboptimality gap for Algorithm \ref{algo:asypa} with different numbers of players over the \textit{log} network.}
  \label{fig:numerics}
\end{figure}

Finally, we show the effectiveness of the aggressive stepsize scheme over the \textit{log} network with $n = 20$ players in Figure \ref{fig:compare}.
While Algorithm \ref{algo:asypa} converges faster than the one without an aggressive stepsize scheme, both algorithms outperform their synchronous version.

\begin{figure}[tbp]
  \centering
  \includegraphics[width=.4\linewidth]{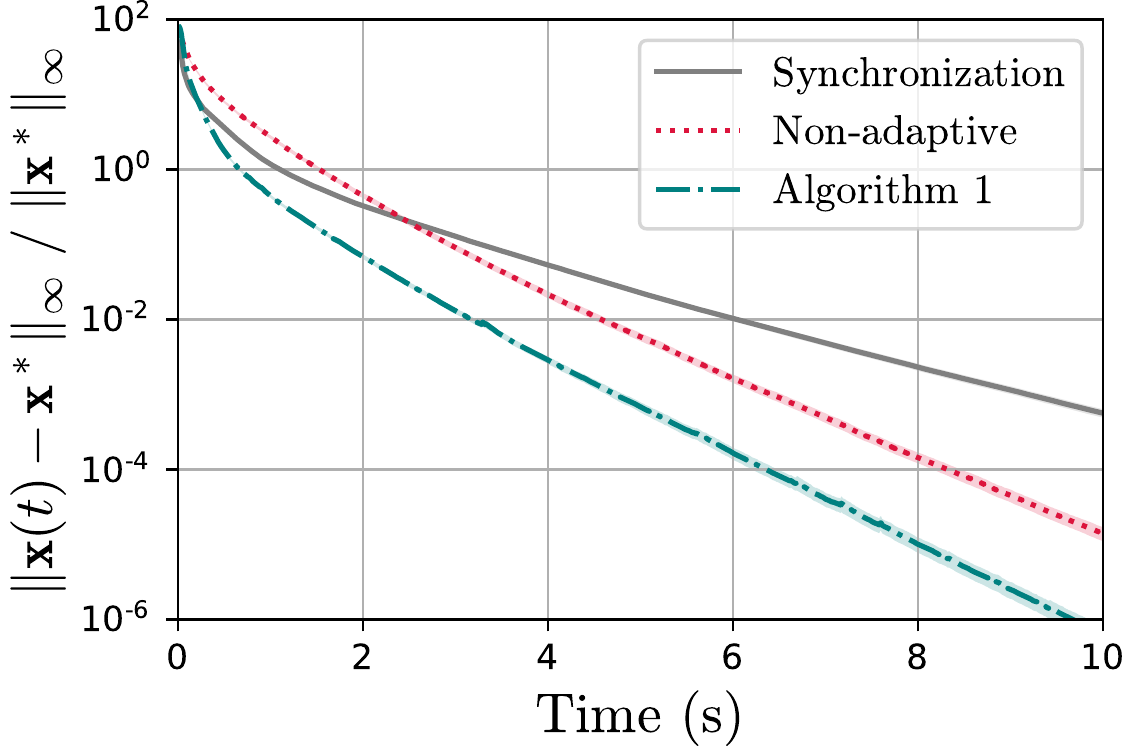}
  \caption{The suboptimality gap versus the running time. Comparison of Algorithm \ref{algo:asypa}, the asynchronous algorithm without an aggressive stepsize scheme (Non-adaptive), and the synchronous algorithm with 20 players over the \textit{log} network.}
  \label{fig:compare}
\end{figure}

\section{Conclusion}
\label{sec:conclusion}

In this paper, we have introduced Asy-NAG, which is challenging due to the full asynchrony of computation and communication.
For the AG, we have proposed a distributed algorithm and proved its convergence to a Nash equilibrium.
It leverages two key ideas: an asynchronous push-sum algorithm to dynamically track the aggregate action, and an aggressive stepsize scheme to facilitate the convergence.
Moreover, we have developed an augmented system approach to address the asynchronicity and designed a novel perturbed coordinated algorithm to help prove the convergence.
Theoretical results have been validated by numerical experiments.

Future work will consider the algorithm design for the Asy-NAGs with linear coupling constraints and the fully asynchronous non-aggregative games. The linear convergence rate for the Asy-NAGs with strongly monotone pseudo-gradients is also an open problem.

\appendix

\section{Proof of Proposition \ref{prop:pcpa}}
\label{ap:pcpa}
Before the proof of Proposition \ref{prop:pcpa}, we state and prove a lemma on the convergence of a deterministic sequence. 
\begin{lm}
  \label{lm:msct}
  Let $\{x_k\}$, $\{u_k\}$, $\{w_k\}$ and $\{v_k\}$ be four scalar sequences such that
  \begin{equation}\label{eq:sct}
  x_{k + 1} \leq x_k - u_k + w_k + v_k, \quad k = 0, 1, \cdots
  \end{equation}
  where $\{x_k\}$, $\{u_k\}$, $\{w_k\}$ are nonnegative, $\sumk w_k < \infty$ and $\sumk v_k$ exists. Then, $\{x_k\}$ converges and $\sumk u_k \leq \infty$.
\end{lm}
\begin{proof}The proof is motivated by that of   \citet[Proposition A.4.4]{bertsekas2015convex}. Using the nonnegativity of $\{u_k\}$, we have
 \begin{equation*}
   x_{k+1} \leq x_k + w_k + v_k \leq x_{\bar{k}} + \sum_{l=\bar{k}}^k (w_l + v_l),
 \end{equation*}
 where $\bar{k} \leq k$. By taking upper limit of both sides as $k\to\infty$, we have that
 \begin{equation}\label{appdne}
   \limsup_{k\to\infty}x_k \leq x_{\bar{k}} + \sum_{l=\bar{k}}^{\infty}w_l + \sum_{l=\bar{k}}^{\infty}v_l.
 \end{equation}
 Since $\sumk w_k < \infty$ and $\sumk v_k$ exists, it follows that the left-hand side of \eqref{appdne} is bounded above and that $\lim_{\bar{k}\to\infty}\sum_{l=\bar{k}}^{\infty}v_l = 0$. By taking the lower limit of the right-hand side (RHS) as $\bar{k}\to\infty$, we have
 \begin{equation*}
   0 \leq \limsup_{k\to\infty}x_k \leq \liminf_{\bar{k}\to\infty}x_{\bar{k}} < \infty.
 \end{equation*}
 This implies that $\{x_k\}$ converges to a finite value. By writing \eqref{eq:sct} for the index $k$ set to $0, \cdots, k$, and adding, we have
 \begin{equation*}
   \sum_{l=0}^k u_k \leq x_0 + \sum_{l=0}^k w_l + \sum_{l=0}^k v_l - x_{k+1}.
 \end{equation*}
 Thus, by taking the limit as $k\to\infty$, we have $\sum_{l=0}^{\infty}u_l<\infty$.
\end{proof}

{\em Proof of Proposition \ref{prop:pcpa}}: To simplify notations, let $\sigma = \sigma_1 + \sigma_2$ and $\widetilde{r}(k) = \max_{j\in \mathcal{N}} r_j(k)$.
It follows from Assumption \ref{as:pcpa} that every player updates at least once in $\sigma$ iterations and $\widetilde{r}(k+t) - \tilr(k) \leq t \sigma_2 \leq t\sigma$ for any integer $t \geq 0$.
Let $I_i^k$ be an indicator function which takes value 1 when $s(k) = i$ and 0 otherwise.
We abuse the notation of stepsize by $\alpha_i(k) = I_i^k\sum_{t=r_i(k)}^{r_i(k+1)-1} \rho(t)$. Then, for all $i \in \mathcal{N}$,
\begin{equation}
  \label{eq:step1}
  \begin{aligned}
  \|\x_i(k+\sigma) - \x_i^*\|^2 
  &= \left\|\x_i(k) - \x_i^* + \sum_{u=k}^{k+\sigma-1}[\x_i(u+1) - \x_i(u)]\right\|^2 \\
  &= \|\x_i(k) - \x_i^*\|^2 + \left\|\x_i(k+\sigma) - \x_i(k)\right\|^2 \\
  & \quad + 2\sum_{u=k}^{k+\sigma-1} [\x_i(k) - \x_i^*]^\T[\x_i(u+1) - \x_i(u)].
  \end{aligned}
\end{equation}
Since the action set $\mathcal{X}_i$ is compact and the mapping $F_i$ is continuous, $\|F_i\|$ is bounded.
Besides, the perturbation $\|\ep_i\|$ is bounded under Assumption \ref{as:pcpa}.
Thus, the perturbed mapping $\|\F_i\|$ is bounded. Let $C$ denote its upper bound; then we have that for $u = k+1, \cdots, k + \sigma$,
\begin{equation*}
  \label{eq:xMinus}
  \begin{aligned}
  \|\x_i(u) - \x_i(k)\| = \left\|\sum_{v=k}^{u-1} [\x_i(v + 1) - \x_i(v)]\right\| \leq C\sum_{t=r_i(k)}^{r_i(u)-1} \rho(t) \leq C\sum_{t=\tilr(k)-\sigma}^{\tilr(k+\sigma)-1} \rho(t).
  \end{aligned}
\end{equation*}
Thus, it holds that
\begin{equation*}
  \begin{aligned}
  \widetilde{\rho}_{1i}(k) \coloneqq \|\x_i(k+\sigma) - \x_i(k)\|^2 &\leq C^2\left[\sum_{t=\tilr(k)-\sigma}^{\tilr(k+\sigma)-1} \rho(t)\right]^2 \leq C^2\sigma(\sigma+1)\sum_{t=\tilr(k)-\sigma}^{\tilr(k+\sigma)-1} \rho^2(t).
  \end{aligned}
\end{equation*}
It follows from Assumption \ref{as:pcpa}(a) that $\tilr(k+\sigma^2 + \sigma) - \sigma \geq \tilr(k + \sigma)$. 

Taking summation over $k$, it yields that
\begin{equation*}
  \begin{aligned}
  \sumk\widetilde{\rho}_{1i}(k) \leq C^2\sigma(\sigma+1)\sumk\sum_{t=\tilr(k)-\sigma}^{\tilr(k+\sigma)-1} \rho(t)^2 \leq C^2\sigma^2(\sigma+1)^2\sum_{t=0}^{\infty}\rho(t)^2 < \infty.
  \end{aligned}
\end{equation*}

Next, consider the last term of the RHS of \eqref{eq:step1}, we obtain that
\begin{equation}
  \label{eq:crossTerm}
  \begin{aligned}
  & \sum_{u=k}^{k+\sigma-1} [\x_i(k) - \x_i^*]^\T [\x_i(u + 1) - \x_i(u)] \\
  &\leq \sum_{u=k}^{k+\sigma-1} [\x_i(u) - \x_i^*]^\T[\x_i(u + 1) - \x_i(u)] + \sum_{u=k}^{k+\sigma-1} \|\x_i(k) - \x_i(u)\|\|\x_i(u + 1) - \x_i(u)\|.
  \end{aligned}
\end{equation}
Denote the last term of the RHS of \eqref{eq:crossTerm} by $\widetilde{\rho}_{2i}(k)$. 
Similarly, we have $\sumk \widetilde{\rho}_{2i}(k) < \infty$. Consider the first term, when $I_i^u = 1$; we then have that
\begin{equation}
  \label{eq:step22}
  \begin{aligned}
  & [\x_i(u) - \x_i^*]^\T[\x_i(u + 1) - \x_i(u)] \\
  &= -\|\x_i(u + 1) - \x_i(u)\|^2 + [\x_i(u + 1) - \x_i^*]^\T[\x_i(u + 1) - \x_i(u) + \alpha_i(u)\widehat{F}_i(u)] \\
  & \quad + \alpha_i(u)[\x_i(u) - \x_i(u + 1)]^\T\widehat{F}_i(u) - \alpha_i(u)[\x_i(u) - \x_i^*]^\T\widehat{F}_i(u) \\
  &\leq - \alpha_i(u)[\x_i(u) - \x_i^*]^\T\widehat{F}_i(u) + \alpha_i(u)\|\x_i(u) - \x_i(u + 1)\|\|\widehat{F}_i(u)\|,
  \end{aligned}
\end{equation}
where the inequality follows from the projection theorem \citep[Proposition 1.1.4]{bertsekas2016nonlinear}.
Similarly, let $\widetilde{\rho}_{3i}(k)$ denote the last term of \eqref{eq:step22}, and we have $\sumk \widetilde{\rho}_{3i}(k) < \infty$.

Considering the first term of the RHS of \eqref{eq:step22}, we obtain
\begin{equation}
  \label{eq:mono_term}
  \begin{aligned}
    &-\alpha_i(u)[\x_i(u) - \x_i^*]^\T\F_i(u) \\
    &= -\alpha_i(u)[\x_i(u) - \x_i(k)]^\T\F_i(u) -\alpha_i(u)[\x_i(k) - \x_i^*]^\T[\F_i(u) - F_i(u)] \\
    & \quad -\alpha_i(u)[\x_i(k) - \x_i^*]^\T[F_i(u) - F_i(k)] -\alpha_i(u)[\x_i(k) - \x_i^*]^\T[F_i(k) - F_i^*] \\
    & \quad -\alpha_i(u)[\x_i(k) - \x_i^*]^\T F_i^*\\
    &\leq  -\alpha_i(u)[\x_i(k) - \x_i^*]^\T[F_i(k) - F_i^*] \\
    & \quad + C\alpha_i(u)\|\x_i(u) - \x_i(k)\| + 2M\alpha_i(u)\|F_i(u) - F_i(k)\| + 2M\alpha_i(u)\|\ep_i(u)\|,
  \end{aligned} 
\end{equation}
where we have used $F_i(u)$ and $F_i^*$ for $F_i(\x_i(u), \overline{\x}(u))$ and $F_i(\x_i^*, \overline{\x}^*)$ for simplicity, and $M = \max_{i \in \mathcal{N}, \x\in \X_i}\|\x\|$.
Note that with the Lipschitz continuity of the mapping $F_i$, we have that
\begin{equation*}
  \begin{aligned}
    &\|F_i(u) - F_i(k)\| \\
    & \leq \|F_i(\x_i(u), \overline{\x}(u)) - F_i(\x_i(u), \overline{\x}(k))\| + \|F_i(\x_i(u), \overline{\x}(k)) - F_i(\x_i(k), \overline{\x}(k))\| \\
    &\leq \frac{L}{n} \sum_{j=1}^n\|\x_j(u) - \x_j(k)\| + L\|\x_i(u) - \x_i(k)\|.
  \end{aligned}
\end{equation*}
Thus, denoting by $\widetilde{\rho}_{4i}(k)$ the summation over $u$ of the last three terms of \eqref{eq:mono_term}, we have $\sumk \widetilde{\rho}_{i4}(k) < \infty$.

Finally, we evaluate the summation over $i$ and $u$ of the first term of the RHS of \eqref{eq:mono_term}. 

By the definition of $I_i^u$ and $\alpha_i(u)$, we know that
\begin{equation*}
  \sumi\sum_{u=k}^{k+\sigma-1}\alpha_i(u) = \sum_{i=1}^n \sum_{t=r_i(k)}^{r_i(k+\sigma)-1}\rho(t).
\end{equation*}
Then, it implies that
\begin{equation}
  \label{eq:step24}
  \begin{aligned}
    &\sumi\sum_{u=k}^{k+\sigma-1}-\alpha_i(u)[\x_i(k) - \x_i^*]^\T[F_i(k) - F_i^*] \\
    &= \sum_{i=1}^n\left[-\sum_{t=r_i(k)}^{\widetilde{r}(k)-1} - \sum_{t=\widetilde{r}(k)}^{\widetilde{r}(k+\sigma)-1} + \sum_{t=r_i(k+\sigma)}^{\widetilde{r}(k+\sigma)-1}\right]
      \times \rho(t)[\x_i(k) - \x_i^*]^\T[F_i(k) - F_i^*] \\
    & \leq -\sum_{t=\widetilde{r}(k)}^{\widetilde{r}(k+\sigma)-1} \rho(t)[\x(k) - \x^*]^\T[\phi(\x(k)) - \phi(\x^*)]
      - \sum_{i=1}^{n}\sum_{t=r_i(k)}^{\widetilde{r}(k)-1}\rho(t)[\x_i(k) - \x_i^*]^\T[F_i(k) - F_i^*] \\
    & \quad + \sum_{i=1}^{n}\sum_{t=r_i(k+\sigma)}^{\widetilde{r}(k+\sigma)-1} \rho(t)[\x_i(k+\sigma) - \x_i^*]^\T[F_i(k+\sigma) - F_i^*] \\
    & \quad + \sum_{i=1}^{n}\sum_{t=r_i(k+\sigma)}^{\widetilde{r}(k+\sigma)-1}\rho(t)\Big\{\|\x_i(k) - \x_i(k+\sigma)\|\|F_i(k) - F_i^*\| + \|\x_i(k+\sigma) - \x_i^*\|\|F_i(k) - F_i(k+\sigma)\|\Big\} \\
    & = -\sum_{t=\widetilde{r}(k)}^{\widetilde{r}(k+\sigma)-1} \rho(t)[\x(k) - \x^*]^\T[\phi(\x(k)) - \phi(\x^*)] - \delta(k) + \delta(k+\sigma) + \widetilde{\rho}_5(k).
  \end{aligned}
\end{equation}
Similarly, we may obtain that $\sum_{k=0}^{\infty} \widetilde{\rho}_5(k) < \infty$.
Besides,
\begin{equation*}
  \begin{aligned}
    \delta(k) &= \sum_{i=1}^{n}\sum_{t=r_i(k)}^{\widetilde{r}(k)-1}\rho(t)[\x_i(k) - \x_i^*]^\T[F_i(k) - F_i^*], \\
    |\delta(k)| &\leq \sum_{i=1}^{n}\sum_{t=\widetilde{r}(k)-\sigma}^{\widetilde{r}(k)-1}2C\rho(t)\|\x_i(k)-\x_i^*\| \leq 4nMC\sigma \rho(\widetilde{r}(k) - \sigma),
  \end{aligned}
\end{equation*}
which implies that $\delta(k)$ is bounded and $\lim_{k\to\infty}\delta(k) = 0$. Thus, it immediately holds that
\begin{equation*}
  \sumk(\delta(k+\sigma)-\delta(k)) = -\sum_{k=0}^{\sigma-1}\delta(k).
\end{equation*}

Finally, combining \eqref{eq:step1}, \eqref{eq:crossTerm}, \eqref{eq:step22}, \eqref{eq:mono_term} and \eqref{eq:step24} yields the following key inequality
\begin{equation}
  \label{eq:keyx}
  \begin{aligned}
    \|\x(k+\sigma)-\x^*\|^2 \leq \|\x(k)-\x^*\|^2 - 2\sum_{t=\widetilde{r}(k)}^{\widetilde{r}(k+\sigma)-1} \rho(t)[\x(k) - \x^*]^\T[\phi(\x(k)) - \phi(\x^*)] + \widetilde{\rho}(k) + \widetilde{\delta}(k),
  \end{aligned}
\end{equation}
where $\widetilde{\rho}(k) = \sumi\big[\widetilde{\rho}_{1i}(k) + 2\sum_{j=2}^{4}\widetilde{\rho}_{ji}(k)\big]+2\widetilde{\rho}_5(k)$ and $\widetilde{\delta}(k) = 2[\delta(k+\sigma)-\delta(k)]$ satisfying that
\begin{equation*}
  \begin{aligned}
    \widetilde{\rho}(k) > 0, \quad \sumk\widetilde{\rho}(k) < \infty, \quad \sumk \widetilde{\delta}(k) \text{ exists}.
  \end{aligned}
\end{equation*}

Letting $\y(k) = \x(k\sigma)$, it follows from \eqref{eq:keyx} that
\begin{equation}
  \label{eq:keyy}
  \begin{aligned}
    \|\y(k+1)-\x^*\|^2 \leq \|\y(k)-\x^*\|^2 - 2\sum_{t=\widetilde{r}(k\sigma)}^{\widetilde{r}((k+1)\sigma)-1} \rho(t)[\y(k) - \x^*]^\T[\phi(\y(k)) - \phi(\x^*)] + \widetilde{\rho}'(k) + \widetilde{\delta}'(k),
  \end{aligned}
\end{equation}
where $\widetilde{\rho}'(k) = \widetilde{\rho}(k\sigma)$, and $\widetilde{\delta}'(k) = \widetilde{\delta}(k\sigma)$. We still have
\begin{equation*}
  \begin{aligned}
    \widetilde{\rho}'(k) > 0, \quad \sumk\widetilde{\rho}'(k) < \infty, \quad \sumk \widetilde{\delta}'(k) \text{ exists}.
  \end{aligned}
\end{equation*}
As \eqref{eq:keyy} satisfies the conditions of Lemma \ref{lm:msct}, it follows that $\{\|\y(k) - \x^*\|\}$ is a convergent sequence and
\begin{equation*}
  \sumk \sum_{t=\widetilde{r}(k\sigma)}^{\widetilde{r}((k+1)\sigma)-1} \rho(t)[\y(k) - \x^*]^\T[\phi(\y(k)) - \phi(\x^*)] < \infty.
\end{equation*}

With similar statements as those in \citet[Proposition 2]{koshalDistributedAlgorithmsAggregative2016}, the sequence $\{\y(k)\}$ converges to $\x^*$. Together with the fact that $\lim_{t\to\infty}\rho(t)=0$, it implies that $\{\x(k)\}$ converges to $\x^*$. \qed

\section{Proof of Lemma \ref{lm:consensus}}
\label{ap:consensus}

Note that $\widetilde{A}(k)$ is a column-stochastic matrix, i.e., $\bone^\T\widetilde{A}(k) = \bone^\T$ for all $k \geq 0$. Lemma \ref{lm:average} below shows that the summation of $\vv_i(k)$ is exactly equal to that of $\x_i(k)$.

\begin{lm}
  \label{lm:average}
  Let $\widetilde{A}(k)$ be such that $\sumi[\widetilde{A}(k)]_{ij} = 1$ for every $j$ and $k\ge 0$. Then,
  \begin{equation}
    \label{eq:average}
    \sum_{i=1}^{\tilde{n}}\vv_i(k) = \sumi \x_i(k),
  \end{equation}
  where $\{\vv_i(k)\}$ and $\{\x_i(k)\}$ are two sequences generated by \eqref{eq:pushsum}.
\end{lm}
\begin{proof}
  From the initialization step of Algorithm \ref{algo:asypa}, \eqref{eq:average} holds for $k=0$. Taking the summation over $i$ on the both sides of the last step in \eqref{eq:pushsum}, we have that
  \begin{equation*}
    \begin{medsize}
    \begin{aligned}
      \sum_{i=1}^{\tilde{n}}\vv_i(k+1) &= \sum_{i=1}^{\tilde{n}}\sum_{j=1}^{\tilde{n}}[\widetilde{A}(k)]_{ij}\vv_j(k) + \sumi [\x_i(k+1) - \x_i(k)] \\
      &= \sum_{j=1}^{\tilde{n}}\vv_j(k) + \sumi \x_i(k+1) - \sumi \x_i(k).
    \end{aligned}
  \end{medsize}
  \end{equation*}
  The result \eqref{eq:average} is then arrived at by induction.
\end{proof}

Defining 
$
\Phi(k, t) = \widetilde{A}(k)\widetilde{A}(k-1)\cdots\widetilde{A}(k-t),
$
Lemma \ref{lm:transMatrix} shows that it converges to a rank one matrix with identical columns, and that $\widetilde{y}_i(k)$ is positive for all $i\in\mathcal{N}$ and $k\geq 0$.

\begin{lm}[\citealp{zhangAsySPAExactAsynchronous2019}]\label{lm:transMatrix}
	Under Assumption \ref{as:graph}, the following statements are in force.
	\begin{enumerate}
		\item There exists a nonnegative vector $\phi(k)$ satisfying $\bone^\T\phi(k)=1$ and
    \begin{equation}
      \left|[\Phi(k,t)]_{ij} - \phi_i(k)\right|\leq B'\lambda^{t}
    \end{equation}
    for all $k\geq t \geq 0$, where
    \begin{equation}
      \textstyle
      B' = 4(1+n^{nb}),\ \lambda=\left(1-\frac{1}{n^{nb}}\right)^{\frac{1}{nb}}
    \end{equation}
    and $b$ is as defined in Lemma \ref{lm:delays}(c).
		\item $\widetilde{y}_i(k) = \sum_{j=1}^n[\Phi(k,k)]_{ij}\geq {n^{-nb}},\ \forall i\in\mathcal{N},k \geq 0.$
	\end{enumerate}
\end{lm}

\textit{Proof of Lemma \ref{lm:consensus}.}
\begin{enumerate}
  \item \label{pf:consensusA} By Lemma \ref{lm:average}, we have that $\bar{\x}(k) = \frac{1}{\tilde{n}}\sum_{i=1}^{\tilde{n}}\widetilde{\vv}_i(k)$. Then, the proof is similar to that of \citet[Lemma 1]{nedicDistributedOptimizationTimeVarying2015}.
  \item It follows from \eqref{eq:pushsum} and the projection theorem that
    \begin{equation*}
      \|\x_i(t+1) - \x_i(t)\| \leq \alpha_i(t) \|F_i(\x_i(t), \z_i(t+1))\|.
    \end{equation*}
    It follows from Assumption \ref{as:FL} that
    $$ 
      \begin{aligned}
      \|F_i(\x_i(t), \z_i(t+1))\| &\leq \|F_i(\x_i(t), \bar{\x}(t))\| + L\|\z_i(t + 1) - \bar{\x}(t)\|.
      \end{aligned}
    $$
    Thus, $\|F_i(\x_i(t), \z_i(t+1))\|$ is bounded under Assumption \ref{as:f} and (a).
    Besides, it follows from Lemma \ref{lm:delays}(c) and Assumption \ref{as:pcpa}(c) that $\lim_{t \to \infty}\alpha_i(t) = 0$. Thus, we have that $\|\x_i(t + 1) - \x_i(t)\| \to 0$ for all $i \in \mathcal{N}$. Then, the rest of the proof and the proof of \ref{lm:consensusD} are similar to that of \citet[Lemma 1]{nedicDistributedOptimizationTimeVarying2015}. \qed
\end{enumerate}

\bibliographystyle{agsm}
\bibliography{ref}

\end{document}